%% file: arxiv_Revised_cayley_plane.tex
\documentclass[11pt,letterpaper]{amsart}

\usepackage{amsmath,amscd}
\usepackage{booktabs,diagbox}

\usepackage{setspace}
\usepackage{multicol}
\usepackage[lite]{amsrefs} 
\usepackage{amssymb} 
\usepackage{graphicx} 
\usepackage{euscript}
\usepackage{color}
\usepackage{array}
\usepackage{picture}
\usepackage{epic}
\usepackage{tikz}
 \usetikzlibrary{backgrounds,cd}
\usepackage{comment}
\usepackage{enumerate}
\usepackage{url}
\usepackage[colorlinks=true, linkcolor=black, citecolor=black, urlcolor=black]{hyperref}
\usepackage{fancyhdr} 
\usepackage{mathtools}
\usepackage{bookmark}
\usepackage[margin=1.36in]{geometry}
\usepackage{caption}
\usepackage{MnSymbol}
\usepackage{enumitem}
\usepackage{braket}

\usepackage{xcolor}
\usepackage{rotating}
\usepackage[all]{xy}
\usepackage{comment}
\usepackage{pdflscape}
\usepackage{tikz}
\usepackage{tikz-cd}
\usepackage{soul,color}
\usepackage[colorinlistoftodos]{todonotes}

\usepackage{tabmacD}

\newcommand*{\fullref}[1]{\hyperref[{#1}]{\ref*{#1}. \nameref*{#1}}}

\tikzset{
  symbol/.style={
    draw=none,
    every to/.append style={
      edge node={node [sloped, allow upside down, auto=false]{$#1$}}
    }
  }
}

\include{pack}

\numberwithin{equation}{section}

\makeatletter
\@namedef{subjclassname@2020}{%
  \textup{2020} Mathematics Subject Classification}
\makeatother

\begin{document}
	\title[Disconnectedness of the Hilbert Schemes of $E_6/P_6$]{Disconnectedness of the Hilbert Schemes of $E_6/P_6$}

 \thanks{Corresponding author: Minyoung Jeon}
\author{Minyoung Jeon}
	  \address{Department of Mathematics, Seoul Women’s University, Seoul 01797, Republic of Korea }\email{\url{minyoung.jeon@swu.ac.kr}}

\date{\today}

\begin{abstract}
In this note, we show that the Hilbert scheme $\text{Hilb}_{P_{d,4}(t)}(E_6/P_6)$ associated with the Hilbert polynomial $P_{d,4}(t)$ is disconnected by determining that it has exactly two connected components. This result extends the disconnectedness of Grassmannians by Seong to the exceptional type $E_6$.

\end{abstract}
\maketitle 
\setcounter{tocdepth}{2}

\setcounter{tocdepth}{1}

\section{Introduction}
  
Let $X$ be a projective variety. Given a fixed polynomial $P(t)$, the Hilbert scheme $\text{Hilb}_{P(t)}(X)$ parametrizes subschemes of $X$ having Hilbert polynomial $P(t)$. The study of the connectedness of such Hilbert scheme is originated in 1966 by Hartshorne \cite{Hart66} that the Hilbert scheme $\text{Hilb}_{P(t)}(\mathbb{P}^n)$ of subschemes of projective space $\mathbb{P}^n$ is connected. In fact, this connectivity does not universally extend to rational homogeneous varieties. Specifically, when $X$ is an ordinary Grassmannian $Gr(k,n)$ for $1<k<n-1$ \cite{Seong20} or an orthogonal Grassmannian \cite{Jeon26}, its Hilbert scheme $\text{Hilb}_{P(t)}(X)$ is known to be disconnected under certain conditions. However, the geometry of Hilbert schemes associated with exceptional types remains an open area of study. The goal of this note is to investigate whether this property occurs in exceptional types. 

By focusing on the Cayley plane, which is an $E_6$-type Grassmannian, we show that its Hilbert scheme $\text{Hilb}_{p(t)}(E_6/P_6)$ is not necessarily connected.

\begin{introthm}
Let $d\geq 3$. The Hilbert scheme $\text{Hilb}_{P_{d,4}(t)}(E_6/P_6)$ consists of exactly two connected components. In particular, it is disconnected.
\end{introthm}
This result is developed in \S\ref{sec3}, where the exact number of components is given in Theorem \ref{t:3.2} and the disconnectedness is in Corollary \ref{c:discon}. For the proof, we extend the work of \cite{Seong20} on the ordinary Grassmannian by combining their methodology with Schubert calculus techniques specific to the $E_6$-type Grassmannian.

A natural direction for future research is to investigate the Nef cone of the Hilbert scheme $\text{Hilb}_{P_{d,4}(t)}(E_6/P_6)$ using Theorem \ref{t:bundle} on the geometry of the Hilbert scheme, analogous to the work of \cite{Seong} for ordinary Grassmannians and \cite{Jeon26} for orthogonal Grassmannians.

Before proceeding to these proofs in \S\ref{sec3}, we recall the necessary background for the $E_6$-type Grassmannian in \S\ref{sec2}.

\section{Background on Flag varieties for $E_6$}\label{sec2}

We review the description of the flag varieties for an Albert algebra and some basic notions necessary for \S\ref{sec3}. The main references for this section are \cites{Gari01,SV00}.

\subsection{Albert algebras}
Let $(\mathfrak{C},\star)$ be the split para-octonion algebra with the multiplication $\star$ on $\mathfrak{C}$, see \cite[\S 2]{Zhao25} for the precise definition of the split para-octonion. 
For a fixed diagonal matrix $\gamma=diag(\gamma_1,\gamma_2,\gamma_3) \in GL_3(\mathbb{C})$, we let $J=H(\mathfrak{C};\gamma)$ be the set of $\gamma$-hermitian $3\times 3$ matrices
\[
x=h(\xi_1,\xi_2,\xi_3;c_1,c_2,c_3)=\begin{bmatrix}
\xi& c_3& \gamma_1^{-1}\gamma_3\overline{c_2}\\
\gamma_2^{-1}\gamma_1\overline{c_3}&\xi_2&c_1\\
c_2&\gamma_3^{-1}\gamma_2\overline{c_1}&\xi_3
\end{bmatrix}
\]
with $\xi_i\in\mathbb{C}$ and $c_i\in \mathfrak{C}$. Here we denote by $\overline{a}$ the conjugation of $a$. A product in $J$ is given by 
\[
xy=\dfrac{1}{2}(x\cdot y+y\cdot x)
\]
where the dot denotes the standard matrix product. We note that the multiplication is not associative. The algebra J is called an {\it Albert algebra}, i.e., a $27$-dimensional central simple exceptional Jordan algebra. 

Then $J$ is naturally equipped with a cubic norm $N: J \rightarrow \mathbb{C}$. We denote by $\text{Inv}(J)$ the group of linear automorphisms preserving the norm of $J$, which is the set of invertible linear maps $\phi: J \rightarrow J$ such that $N(\phi(x)) = N(x)$ for all $x \in J$. For brevity, these maps $\phi$ will be referred to as norm isometries. This group $G=\text{Inv}(J)$ precisely realizes the simply connected complex algebraic group of type $E_6$.

Given a cubic norm $N$ on $J$, we define the linear map $f_x:J\rightarrow \mathbb{C}$ such that $f_x(y)$ is the coefficient of the variable $t$ in the expansion of $N(x+ty)$ for $x,y \in J$. As $J$ is endowed with a linear trace map $T:J\rightarrow \mathbb{C}$, we have a non-degenerate form $\beta:J\times J\rightarrow \mathbb{C}$ defined by $\beta(x,y):=T(xy)$. A quadratic map $\#:J\rightarrow J, x\mapsto x^\#$ is defined uniquely by 
\[
\beta(x^\#,y)=f_x(y)
\]
for all $y\in J$, satisfying $(x^\#)^\#=N(x)x$. The {\it Frueudenthal cross product} is a linearization of $\#$ given by 
\[
x\times y:=(x+y)^\#-x^\#-y^\#.
\]

\subsection{Flag varieties}

Let $G$ be a simple linear algebraic group over $\mathbb{C}$, $T\subset G$ the maximal torus, and $B(\supset T)$ a Borel subgroup of $G$. Let $\Delta=\{\alpha_1,\ldots,\alpha_n\}$ be the set of simple roots, where $n$ is the rank of $G$. For any set of indices $I\subset\{1,\ldots,n\}$, we denote by $P_I$ the standard parabolic subgroup obtained by omitting simple roots $\{\alpha_j\mid j\in I\}$. The flag variety associated with $I$ is given by 
\[
Fl(I):=G/P_I.
\]

We consider the group $G=\text{Inv}(J)$ of type $E_6$ with an Albert algebra $J$. We index the simple roots of the Dynkin diagram as
\[
\begin{tikzcd}
&&\overset{2}{\circ}&&\\
\underset{1}{\circ}\arrow[dash,r]&\underset{3}{\circ}\arrow[dash,r]&\underset{4}{\circ}\arrow[dash,u]\arrow[dash,r]&\underset{5}{\circ}\arrow[dash,r]&\underset{6}{\bullet}.
\end{tikzcd}
\]

We say an element $x\in J$ is of rank one if $x^\#=0$. A subspace $V$ of $J$ is {\it totally singular} if all elements of $V$ are of rank one.

For $E_6=Inv(J)$, each $Fl(i)=E_6/P_i$ corresponds to $i$-spaces, as subspaces of $J$. The $1$-spaces, $3$-spaces and $4$-spaces are totally singular subspace of $J$ with dimensions $1,2,$ and $3$, respectively. The $5$-spaces are the $5$-dimensional maximal totally singular subspaces, and the $2$-spaces are the $6$-dimensional totally singular subspaces. Lastly, the $6$-spaces known as the {\it Cayley plane} $E_6/P_6$ are the $10$-dimensional subspaces of the form
\[
x\times J=\{x\times y\mid y\in J\}
\]
for $x\in J$. 
Each $V_i$ represents a point, an $i$-space in $Fl(i)=E_6/P_i$. 

Two $i$-spaces are {\it incident} exactly when they coincide. The incidence relation is given by inclusion except that
\begin{enumerate}
\item $V_2$ and $V_5$ are incidence if and only if $\text{dim}(V_2\cap V_5)=3$
\item $V_2$ and $V_6$ are incidence if and only if $\text{dim}(V_2\cap V_6)=5$.
\end{enumerate}
For the second condition, the intersection space is necessarily a nonmaximal totally singular subspace of $J$.
 Furthermore, the flag variety $Fl(I)$ associated to $I=\{i_1,\ldots,i_m\}$ is expressed as
\[
\left\{ (V_1,\ldots,V_m)\in Fl(i_1)\times\cdots\times Fl(i_m)\mid V_{i_j} \;\text{ is incident to } V_{i_k}\text{ for all }1\leq j,k\leq m\right\}.
\]

For the exceptional group $G=E_6$, $Fl(i)=E_6/P_i$ can be embedded to a projective space by the {\it minimal embedding} $\iota_i:Fl({i})\hookrightarrow\mathbb{P}^\mathcal{N}$, where the dimension $\mathcal{N}$ is given by
\[
\mathcal{N}=\begin{cases}
26& \text{if }i=1,6\\
77& \text{if }i=2\\
350& \text{if }i=3,5\\
2924& \text{if }i=4
\end{cases},
\] 
see \cite[p.304]{OV90}

\subsection{Schubert varieties}
Let $W$ be the Weyl group of $G$ and $\ell:W\rightarrow \mathbb{N}$ the length function. We denote by $w_0$ the longest element. The Weyl group $W_{P}$ of $P=P_I$ is the subgroup of $W$ generated by the simple reflections $\{s_j\mid j\notin I\}$. Let $W^P$ denote the set of minimal length representatives for the cosets in $W/W_P$. 

For any $w \in W^P$, the Schubert variety $X_{w,P}$ of dimension $\ell(w)$ and the opposite Schubert variety $\Omega_{w,P}$ of dimension $\text{dim}(G/P) - \ell(w)$ are defined as the closures of the $B$-orbit $X_{w,P}^\circ = BwP/P$ and the $B^-$-orbit $\Omega_{w,P}^\circ = B^-wP/P$, respectively.

Let $\Phi^+$ and $\Phi^-$ be the positive and negative roots. As stated in \cite[Proposition 2.1, lemma 2.2]{TY09}, there is a bijection between the element $w\in W^P$ and the partition $\lambda:=\lambda(w)$, where the partition $\lambda$ is determined by the inversion set
\[
I(w)=\{\alpha \in \Phi^+\mid w\cdot \alpha\in \Phi^-\}.
\]
For instance, in our case of $E_6/P_6$, the partitions of the longest element $w_{\circ,P_6}$ and $w=s_4s_2s_1s_3s_4s_5s_6$ in $W^{P_6}$ are $(1,1,2,4,4,2,1,1)$ and $(1,1,2,2,1)$, respectively. They are depicted as the following diagram.

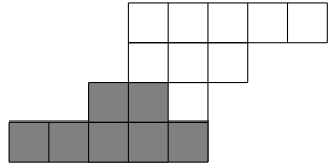
\[
\begin{picture}(51,65)
{
\put(-5,10){\color{gray}\rule{10.5mm}{10.5mm}}
\put(-35,10){\color{gray}\rule{26.5mm}{5.5mm}}

\multiput(10,55)(0,15){2}{\line(1,0){75}}
\multiput(10,55)(15,0){6}{\line(0,1){15}}

\multiput(10,40)(0,15){2}{\line(1,0){45}}
\multiput(10,40)(15,0){4}{\line(0,1){15}}

\multiput(-5,25)(0,15){2}{\line(1,0){45}}
\multiput(-5,25)(15,0){4}{\line(0,1){15}}

\multiput(-35,10)(0,15){2}{\line(1,0){75}}
\multiput(-35,10)(15,0){6}{\line(0,1){15}}
}
\end{picture}
\]
Note that the gray shaded regions indicate the partition for $w=s_4s_2s_1s_3s_4s_5s_6$. The dual (or complement) partition $\lambda^c$ is obtained via the {\it rotate} operation \cite[\S2.2]{TY09}, which is a $\pi$-rotation for the $E_6$ partition. In the example of $\lambda=(1,1,2,2,1)$ with $w_\lambda=s_4s_2s_1s_3s_4s_5s_6$, the dual partition is $(1,1,2,3,2)$ which corresponds to $w_\circ w_\lambda w_\circ^{-1} w_{\circ,P_6}$. Here $w_0$ is the longest element in $W$. We refer the reader to \cite[\S2.2]{TY09} for the precise correspondence between the elements $w\in W^P$, their associated partitions $\lambda(w)$, and the dual partitions.

\subsection{The Hilbert scheme}

For $r\geq 2$ and $d\geq 3$, and $N\geq r$, the Hilbert polynomial of a degree $d$ hypersurface $X\subset\mathbb{P}^r$ is $P_{d,r}(t)$, where
\[
P_{d,r}(t):=\binom{t+r}{r}-\binom{t+r-d}{r}.
\]

We can effectively study the Hilbert scheme $\mathrm{Hilb}_{P_{d,r}(t)}(G/P_{i})$ by identifying it with the space of degree $d$ hypersurfaces in $G/P_{i}$, as established in the following theorem.

\begin{thm}[ {\cite[Theorem 2.1]{Jeon26}} ]\label{t:2.1}
The space of degree $d$ hypersurfaces of $\mathbb{P}^r$ in $G/P_{i}$ is in bijection with the Hilbert scheme $\mathrm{Hilb}_{P_{d,r}(t)}(G/P_{i})$.
\end{thm}

\section{Disconnectedness of the Hilbert schemes of $E_6/P_6$}\label{sec3}

In this section, we explore the disconnectedness of the Hilbert scheme associated with the Cayley plane $E_6/P_6$, which is famously realized as the octonionic projective plane $\mathbb{OP}^2$. To simplify notation in what follows, we set $\mathcal{G}=E_6/P_6$. Let $d\geq 3$.

We start with a lemma about projective spaces that contain a threefold of $\mathcal{G}$.

\begin{lemma}\label{lem3.1}
Let $X$ be a degree $d$ threefold of $\mathbb{P}^4$ in $\mathcal{G}$. If $L$ is an $4$-dimensional projective space containing the threefold $X$ in $\mathbb{P}^{26}$, 
then the projective space $L$ is contained in $\mathcal{G}$.
\end{lemma}
\begin{proof}
According to \cite[\S6.2]{LM03}, the Cayley plane $\mathcal{G}$ is defined by the quadratic equations $\mathcal{K}_i$ for $i$ in some index set $I$. By the assumption, we have $X\subset L\cap(\cap_i \mathcal{K}_i)$. The rest of the arguments are precisely the same as in \cite[Theorem 3.2]{Seong20}, \cite[Lemma 3.1]{Jeon26}.
\end{proof}

By applying the previous lemma, we obtain the following result regarding the number of connected components of the Hilbert scheme.

\begin{thm}\label{t:3.2}
There are two connected components in the Hilbert scheme $\text{Hilb}_{P_{d,4}(t)}(\mathcal{G})$.
\end{thm}
\begin{proof}
Let $L$ be a projective space of dimension $4$ in $\mathbb{P}^{26}$ containing a degree $d$ hypersurface $X$ in $\mathbb{P}^4\subset \mathcal{G}$.

Using the Pieri's rule for the $E_6$-Grassmannian, $E_6/P_6$ \cite{TY09}, the possible cohomology classes for $L$ are either $\sigma_{\lambda^1}$ or $\sigma_{\lambda^2}$ where the partitions $\lambda^i=\lambda^i(w_i)$ for $i=1,2$ are given by
\[
\lambda^1=(1,1,2,4,4)\quad\text{and}\quad\lambda^2=(1,1,2,4,3,1).
\]
Then the possible cohomology class $[X]$ for the hypersurface $X$ of degree $d$ is 
\begin{equation}\label{e:class for X}
[X]=d\cdot \sigma_\mu, \mu=(1,1,2,4,4,1).
\end{equation}
for both cases $[L]=\sigma_{\lambda^i}$ for $i=1,2$. 

We know from \cite[Theorem 1.1]{HM} that the projective spaces in $E_6$-Grassmannians whose cohomology class is exactly the same as the one for $L$ are related by linear automorphisms on that Grassmannians. So, the set of $4$-dimensional projective spaces with cohomology class $\sigma_{\lambda^i}$ is connected for each $i=1,2$. 

Even though there is the only one possible cohomology class for $X$ as \eqref{e:class for X}, the cohomology classes $\sigma_{\lambda^1}$ and $\sigma_{\lambda^2}$ are different as they are parametrized by two different partitions $\lambda^i$. Hence, by Lemma \ref{l:3.3}, we have that the Hilbert scheme $\text{Hilb}_{P_{d,4}(t)}(\mathcal{G})$ is not connected.
\end{proof}

The proof of Theorem \ref{t:3.2} thus reduces to proving the following lemma.
\begin{lemma}\label{l:3.3}
Let $X$ and $X'$ be degree $d$ threefolds in $\mathcal{G}$ such that both $I_X$ and $I_{X'}$ in $\text{Hilb}_{P_{d,4}(t)}(\mathcal{G})$ are in the same connected components of the Hilbert scheme. If $Q, Q'\cong \mathbb{P}^4$ in $\mathcal{G}$ such that $X\subset Q, X'\subset Q'$, then the cohomology class of $Q$ in $\mathcal{G}$ is equal to the one of $Q'$ in $\mathcal{G}$.
\end{lemma}
\begin{proof}
We adapt the deformation strategy from \cite[Lemma 4.1]{Seong20}. We consider the minimal embedding $\mathcal{G}\rightarrow \mathbb{P}^{26}_\mathbb{C}$. We let $I_Q$ and $I_{Q'}$ be ideals in $\mathbb{P}^{26}$ corresponding to projective $4$-spaces $Q$ and $Q'$. The inclusion $X\subset Q$ and $X'\subset Q'$ induce the inclusions 
\[
I_Q\subset I_X\text{ and } I_{Q'}\subset I_{X'}.
\]

As $I_X$ and $I_{X'}$ are in the same connected component in the Hilbert scheme, one can find a connecting flag family $\{I_t\}$ where $t$ is the deformation parameter so that $I_0=I_X$ and $I_1=I_{X'}$. This flat family parametrizes a deformation from $X$ to $X'$. Since $Q$ and $Q'$ are uniquely determined by the spaces containing these degree $d$ threefolds, it also induces a continuous deformation between $Q$ and $Q'$ in $\mathcal{G}$. This concludes that the cohomology classes for $Q$ and $Q'$ coincide.
\end{proof}

Finally, the disconnectedness of the Hilbert scheme is established by combining Theorem \ref{t:3.2} and Lemma \ref{l:3.3}.
\begin{cor}\label{c:discon}
Let $d\geq 3$. The Hilbert scheme $\text{Hilb}_{P_{d,4}(t)}(\mathcal{G})$ is disconnected. 
\end{cor}

We conclude this manuscript by establishing the geometric description of the Hilbert scheme.

Let $J\in \{\{1,2\},\{3\}\}$ and $I=J\cup \{6\}$. Let $\pi_J \colon \mathrm{Fl}(I) \to \mathrm{Fl}(J)$ be the projection map that drops the $6$-space $V_6$. Similarly, let $\pi_{6} \colon \mathrm{Fl}(I) \to \mathrm{Fl}(6)$ be the natural projection onto $\mathrm{Fl}(6)$. This yields the following diagram:
\begin{equation}\label{e:tits}
\begin{tikzcd}
& Fl(I)\arrow[dl,"\pi_{6}"'] \arrow[dr,"\pi_{J}"] & \\
Fl(6)  & & Fl(J)
\end{tikzcd}
\end{equation}

\begin{thm}\label{t:bundle}
Let $E_{J}\subset V_{\omega_6}\otimes \mathcal{O}_{Fl(J)}$ be the homogenous vector bundle over $Fl(J)$ whose fiber over any point $z\in Fl(J)$ is the linear space $\langle \pi_6(\pi_{J}^{-1}(z))\rangle\subset V_{\omega_6}$. Then the Hilbert scheme $\mathrm{Hilb}_{P_{d,4}(t)}(\mathcal{G})$ is isomorphic to $\mathcal{M}=\mathbb{P}(\mathrm{Sym}^d{\mathcal{E}}^\vee)$, where $\mathcal{E}$ is the bundle over the disjoint union of flag varieties
\[
Fl(1,2)\sqcup Fl(3)
\]
such that 
\[
\mathcal{E} := 
\begin{cases} 
{E}_{1,2} & \text{on }Fl(1,2) \\
{E}_{3} & \text{on } Fl(3)
\end{cases}.
\]
\end{thm}
\begin{proof}
 Let $\mathcal{Z}_{1,2}$ be the family of subvarieties defined by
\[
\mathcal{Z}_{1,2} := \pi_{1,2}^{-1}(y) = \left\{ V_{6} \;\middle|\; 
\begin{array}{l}
 F_{1}\text{ is incident to }F_{2}, \\
F_{1}\text{ is incident to }V_{6}, \\
F_{2}\text{ is incident to }V_{6}
\end{array}
\right\}
\]
for $y=(F_{1},F_{2})\in Fl(1,2)$, and $\mathcal{Z}_{3}$ by
\[
\mathcal{Z}_{3} := \pi_{3}^{-1}(y) = \left\{ V_{6} \;\middle|\; 
\begin{array}{l}
F_{3}\text{ is incident to }V_{6}, \\
\end{array}
\right\}
\]
for $y=(F_{3})\in Fl(3)$.

This family is then parametrized by the flag variety $Fl(J)$. It is known that 
the maps
\begin{align}
 E_6/P_{1,2,6}&\rightarrow E_6/P_{1,2}\label{a:P4}\\
 E_6/P_{3,6}&\rightarrow E_6/P_{3}\label{b:P4}
\end{align}
are $\mathbb{P}^4$-bundles as their fibers are $P_{1,2}/P_{1,2,6}$ and $P_{3}/P_{3,6}$ in $\mathcal{G}:=E_6/P_6$ respectively. This implies that the fiber $\mathcal{Z}_{J}$ can be identified with a projective space, and hence, a family of linear spaces $\mathbb{P}^4$ in $\mathcal{G}$ is parametrized by the flag variety $Fl(J)$. Especially, the cohomology class of $\mathcal{Z}_{J}$ in the Chow ring $A^*(\mathcal{G})$ is given by 
\begin{equation}
[\mathcal{Z}_{J}]=\begin{cases}
\sigma_{\lambda^1} & \text{if } J=\{1,2\}, \\
\sigma_{\lambda^2} & \text{if } J=\{3\},
\end{cases}
\end{equation}
where the
the partitions $\lambda^1$ and $\lambda^2$ are 
\[
\lambda^1=(1,1,2,4,3,1) \text{ and } \lambda^2=(1,1,2,4,4).
\]
The corresponding Schubert varieties are
\[
X_{w^1}(F_1,F_2)=\left\{V_6\mid F_1\subset F_2, F_1\subset V_6, \mathrm{dim}\;(F_2\cap V_6)=5 \right\}\cong P_{1,2}/P_{1,2,6} \cong \mathbb{P}^4\subset E_6/P_6
\]
associated with $w^1=s_3s_4s_5s_6\in W^{P_6}$ for $\lambda^1$, and 
\[
X_{w^2}(F_3)=\left\{V_6\mid  F_3\subset V_6 \right\}\cong P_{3}/P_{3,6} \cong \mathbb{P}^4\subset E_6/P_6
\]
associated with $w^2=s_2s_4s_5s_6\in W^{P_6}$ for $\lambda^2:=\lambda(w^2)$ respectively, where $F_i$ and $V_i$ are $i$-spaces. We note that the stabilizers of $X_{w^1}$ and $X_{w^2}$ are $P_{1,2}$ and $P_3$, respectively.

We can view this structure globally in the following way. By \cite[\S2.7.1]{LM04} and with the diagram \eqref{e:tits}, $Y_z:=\pi_6(\pi_{J}^{-1}(z))\subset Fl(6)=E_6/P_6$ is the {\it Tits transform} of $z\in Fl(J)$. As noted in \cite[\S2.7.3]{LM04}, by assigning each $z\in Fl(J)$ to its corresponding linear subspace $\langle Y_z\rangle\subset V_{\omega_6}$, we can construct a homogeneous vector bundle $E_{J}\rightarrow Fl(J)$ having fibers $\langle Y_z\rangle$. By this construction, $E_{J}$ forms a subbundle of the trivial bundle $V_{\omega_6}\otimes\mathcal{O}_{Fl(J)}$ over $Fl(J)$. Then $Fl(I)$ can be identified with the projectivization of $E_{J}$, so that 
\[
Fl(I)\cong \mathbb{P}(E_{J}).
\]

Let $\mathcal{M}=\mathbb{P}(\mathrm{Sym}^d(E_{J}^\vee))$ be the bundle over $Fl(J)$ for $d\geq 3$. Then the fiber of $\mathcal{M}$ over $z\in Fl(J)$ parametrizes the degree $d$ threefolds in $Z_{J}\cong Y_z\cong \mathbb{P}^4\subset \mathcal{G}$. Therefore, the connected component in $\mathrm{Hilb}_{P_{d,4}(t)}(\mathcal{G})$ which corresponds to the cohomology class $X_{w^i}$ is isomorphic to the bundle $\mathcal{M}\rightarrow Fl(J)$ where $i=1$ if $J=\{1,2\}$ and $i=2$ if $J=\{3\}$.

 We use the standard argument to show that the map $\Phi$ from the bundle $\mathcal{M}$ to the Hilbert scheme $\mathrm{Hilb}_{P_{d,4}(t)}(\mathcal{G})$ is a scheme-theoretic isomorphism.

We first compute the Zariski tangent space of the Hilbert scheme at a point $[X]$, which is given by $H^0(X, \mathcal{N}_{X}\mathcal{G})$ \cite[Theorem 2.4]{Hartshorn_Deform}, where $X$ is a degree $d$ threefold in $\mathcal{G}$.

The standard normal bundle sequence for the inclusions $X \subset \mathbb{P}^4 \subset \mathcal{G}$ results in
\begin{equation}\label{e:exact normal}
0 \to \mathcal{N}_{X}\mathbb{P}^4 \to \mathcal{N}_{X}\mathcal{G} \to \mathcal{N}_{\mathbb{P}^4}\mathcal{G}|_X \to 0,
\end{equation}
see \cite[B.7.4]{Ful}.
Since $\mathcal{N}_{X}\mathbb{P}^4 \cong \mathcal{O}_X(d)$ and $d \ge 3$, we have $h^0(X, \mathcal{N}_{X}\mathbb{P}^4) = \binom{d+4}{4} - 1$ and $H^1(X, \mathcal{N}_{X}\mathbb{P}^4) = 0$, \cite[Example 3.2.8 (ii)]{Sernesi}.

 Furthermore, since deformations of $\mathbb{P}^4$ in $\mathcal{G}$ are unobstructed and parametrized by the smooth base $Fl(J)$, the restricted normal bundle $\mathcal{N}_{\mathbb{P}^4}\mathcal{G}|_X$ is globally generated with $h^0(X, \mathcal{N}_{\mathbb{P}^4}\mathcal{G}|_X) = \dim Fl(J)$, \cite[Corollary 3.2.7]{Sernesi}.
Taking the long exact sequence in cohomology
\[
0 \to H^0(X, \mathcal{N}_{X}\mathbb{P}^4) \to H^0(X, \mathcal{N}_{X}\mathcal{G}) \to H^0(X, \mathcal{N}_{\mathbb{P}^4}\mathcal{G}|_X) \to H^1(X, \mathcal{N}_{X}\mathbb{P}^4) \to \dots
\] 
 deduced from \eqref{e:exact normal}, we obtain
\[
h^0(X, \mathcal{N}_{X}\mathcal{G}) = N_d + \dim Fl(J),
\]
where $N_d:=\binom{d+4}{4} - 1$.
This dimension is exactly equal to the dimension of $\mathcal{M}$. This implies that the local rings of the target Hilbert scheme are regular. So, the corresponding connected component of the Hilbert scheme is smooth and thus normal.

In addition, as $\Phi$ is a bijection on closed points and induces an isomorphism on Zariski tangent spaces, it is unramified. Being a proper, unramified bijection onto a normal scheme, by Zariski's Main Theorem \cite[Corollary 4.3]{Qing}, $\Phi$ is an isomorphism of schemes.
Hence, for each case, the bundle $\mathcal{M}\rightarrow Fl(J)$ is isomorphic to the connected component of the Hilbert scheme $\mathrm{Hilb}_{P_{d,4}(t)}(\mathcal{G})$ linked to $\sigma_{\lambda^i}$. 
\end{proof}

\begin{rmk}
In the proof of Theorem \ref{t:bundle}, we provide additional arguments to deduce the scheme-theoretic isomorphism, a step not explicitly presented in Seong's analogous statement for ordinary Grassmannian cases \cite[Theorem 3.3]{Seong}. We believe this reflects his underlying mathematical perspective.
\end{rmk}

\begin{ack}
We are grateful to Jaehyun Hong and Minseong Kwon for their helpful suggestions in finding the relevant references necessary for this project. We thank the members of the Center for Complex Geometry at the Institute for Basic Science for their valuable support. We also would like to convey our appreciation to Akiyoshi Sannai for a suggestion and pointing out errors in Theorem 3.5, which will be addressed as an example in his upcoming project. 
\end{ack}

\bibliographystyle{alpha}
\bibliography{biblio}
\end{document}

%% file: pack.tex
\usepackage{amsfonts}
\usepackage{amssymb}
\usepackage{extarrows}
\usepackage{cleveref}
\usepackage{amsthm}
\usepackage{amscd}
\usepackage{amsmath}
\usepackage{mathtools}
\usepackage{mathrsfs}

\usepackage{color}	
\definecolor{due}{RGB}{0,76,147}

\usepackage{graphicx}	
\usepackage{multicol}	
\usepackage{wrapfig}
\usepackage{enumitem}

\usepackage{longtable}

\theoremstyle{definition}
\newtheorem{defi}{Definition}[section]
\theoremstyle{plain}
\newtheorem{thm}[defi]{Theorem}

\newtheorem{introthm}{Theorem}[section]

\newtheorem{cor}[defi]{Corollary}
\newtheorem{lemma}[defi]{Lemma}
\theoremstyle{remark}

\newtheorem{rmk}[defi]{Remark}

\theoremstyle{definition}

\newtheorem*{ack}{Acknowledgement}

  \makeatletter
\newcommand{\xdashrightarrow}[2][]{\ext@arrow 0359\rightarrowfill@@{#1}{#2}}
\newcommand{\xdashleftarrow}[2][]{\ext@arrow 3095\leftarrowfill@@{#1}{#2}}
\newcommand{\xdashleftrightarrow}[2][]{\ext@arrow 3359\leftrightarrowfill@@{#1}{#2}}
\def\rightarrowfill@@{\arrowfill@@\relax\relbar\rightarrow}
\def\leftarrowfill@@{\arrowfill@@\leftarrow\relbar\relax}
\def\leftrightarrowfill@@{\arrowfill@@\leftarrow\relbar\rightarrow}
\def\arrowfill@@#1#2#3#4{%
  $\m@th\thickmuskip0mu\medmuskip\thickmuskip\thinmuskip\thickmuskip
   \relax#4#1
   \xleaders\hbox{$#4#2$}\hfill
   #3$%
}
\makeatother
